\documentclass[12pt,reqno]{amsart}
\usepackage{graphicx}
\usepackage{psfrag}
\usepackage{pxfonts}
\usepackage{mathrsfs}
\usepackage{color}
\usepackage{amsmath,amssymb}

\numberwithin{equation}{section}

\newcommand{\x} {\tilde{x}}

\theoremstyle{plain}
\newtheorem{maintheorem}{Theorem}

\newtheorem{theorem}{Theorem}[section]

\newtheorem{cor}[theorem]{Corollary}
\newtheorem{proposition}[theorem]{Proposition}
\newtheorem{lemma}[theorem]{Lemma}
\newtheorem{definition}[theorem]{Definition}

\theoremstyle{remark}
\newtheorem{remark}[theorem]{Remark}

\begin{document}

\thanks{The author thanks Sustech for the work environment and financial support.}

\author[F. Micena]{Fernando Micena}
\address{Instituto de Matem\'{a}tica e Computa\c{c}\~{a}o,
  IMC-UNIFEI, Itajub\'{a}-MG, Brazil.}
\email{fpmicena82@unifei.edu.br}


\renewcommand{\subjclassname}{\textup{2000} Mathematics Subject Classification}

\date{\today}

\setcounter{tocdepth}{2}

\title{On Measurable Properties of Anosov Endomorphisms of Torus}
\maketitle
\begin{abstract} We found a dichotomy involving the unstable Lyapunov exponent of a special Anosov endomorphism of the torus induced by the conjugacy with the linearization.
In fact, either every unstable leaf meets on a set of zero measure the set for which is defined such unstable Lyapunov exponent or the endomorphism is smoothly conjugated with its linearization. Also, we are able to characterize the absolute continuity of the intermediate foliation for a class of volume preserving special Anosov endomorphisms of $\mathbb{T}^3.$
\end{abstract}

\section{Introduction and Statement of the Results}\label{section.preliminaries}

In $1970s,$ the works \cite{PRZ} and \cite{MP}  generalized the notion of Anosov diffeomorphism for non-invertible maps, introducing the notion of Anosov endomorphism. We consider $M$ a $C^{\infty}-$closed Riemannian manifold.

\begin{definition}\cite{PRZ} \label{defprz} Let $f: M \rightarrow M$ be a  $C^1$ local diffeomorphism. We say that $f$ is an Anosov endomorphism if there are constants $C> 0$ and $\lambda > 1,$ such that, for every $(x_n)_{n \in \mathbb{Z}}$ an $f-$orbit there is a splitting

$$T_{x_i} M = E^s_{x_i} \oplus E^u_{x_i}, \forall i \in \mathbb{Z},$$

which is preserved by $Df$ and for all $n > 0 $ we obtain

$$||Df^n(x_i) \cdot v|| \geq C^{-1} \lambda^n ||v||, \;\mbox{for every}\; v \in E^u_{x_i} \;\mbox{and for any} \; i \in \mathbb{Z},$$
$$||Df^n(x_i) \cdot v|| \leq C\lambda^{-n} ||v||, \;\mbox{for every}\; v \in E^s_{x_i} \;\mbox{and for any} \; i \in \mathbb{Z}.$$

\end{definition}

We denote by $M^f$ the space of all $f-$orbits $\x= (x_n)_{n \in \mathbb{Z}},$ endowed with me metric $$\bar{d}(\tilde{x}, \tilde{y}) =  \sum_{i \in \mathbb{Z}} \frac{d(x_i, y_i)}{2^{|i|}},$$ where $d$ denotes the Riemannian metric on $M$ and $\x= (x_n)_{n \in \mathbb{Z}}, \tilde{y}= (y_n)_{n \in \mathbb{Z}},$ two $f-$orbits. We denote by $p: M^f \rightarrow M,$ the natural projection $$p((x_n)_{n \in \mathbb{Z}}) = x_0.$$

The space $(M^f, \bar{d})$ is compact, moreover $f$ induces a continuous map $\tilde{f}: M^f \rightarrow M^f,$ given by the shift
$$\tilde{f}((x_n)_{n \in \mathbb{Z}}) = (x_{n+1})_{n \in \mathbb{Z}}. $$

Anosov endomorphisms can be defined in an equivalent way (\cite{MP}).

\begin{definition}\cite{MP} \label{defmp} A $C^1$ local diffeomorphism $f: M \rightarrow M$ is said an Anosov endomorphism if $Df$ contracts uniformly a $Df-$invariant and continuous sub-bundle $E^s \subset TM$ into itself and the action of $Df$ on the quotient $TM/E^s$ is uniformly expanding.
\end{definition}

\begin{proposition}[\cite{MP}] A local diffeomorphism $f: M \rightarrow M$ is an Anosov endomorphism of $M$ if and only if the lift $\overline{f}: \overline{M} \rightarrow \overline{M}$ is an Anosov diffeomorphism of $\overline{M},$ the universal cover of $M.$
\end{proposition}

Sakai, in \cite{SA} proved that, in fact, the definitions $\ref{defprz}$ and $\ref{defmp}$ are equivalent.  An advantage to work with the definition given in \cite{MP} is that in $\overline{M}$ there are invariant foliations $\mathcal{F}^s_{\overline{f}}$ and $\mathcal{F}^u_{\overline{f}}.$  So in the universal cover we can use good properties of these foliations as absolute continuity and quasi-isometry, for instance.

Let $f: M  \rightarrow M$ be a $C^r-$Anosov endomorphism with $r \geq 1,$ it is known that $E^s_f(\tilde{x})$ and $E^u_f(\tilde{x})$ admit uniform size local $C^r-$tangent sub manifolds $W^s_f(\tilde{x})$ and $W^u_f(\tilde{x}).$
\begin{enumerate}
\item $W^s_f(x) = \{y \in M \;|  \displaystyle\lim_{n \rightarrow +\infty} d(f^n(x), f^n(y)) = 0\},$
\item $W^u_f(\tilde{x}) = \{y \in M \;| \exists \tilde{y} \in M^f \; \mbox{such that}\; y_0 = y \; \mbox{and} \;  \displaystyle\lim_{n \rightarrow +\infty} d(x_{-n}, y_{-n}) = 0\}.$
\end{enumerate}

The leaves $W^s_f(\tilde{x})$ and $W^u_f(\tilde{x})$  vary  $C^1-$continuously with $\tilde{x},$ see Theorem 2.5 of \cite{PRZ}.

It is known by \cite{PRZ} and \cite{MP} that structural stability can fail for Anosov endomorphisms. It is because unstable directions depend on the past orbits of a point. For Anosov (non-invertible) endomorphisms, a given point $x$ can have uncountable past orbits, and so it has uncountable unstable manifolds passing by $x,$ see \cite{PRZ}. When an Anosov endomorphism is such that for any point is defined a unique unstable direction, independently of its past orbits, such endomorphism is called special Anosov endomorphism.

Here our focus is Anosov endomorphisms of the torus. Given $f$ and Anosov endomorphism of torus, we consider $A = f_{\ast}: \mathbb{Z}^2 \rightarrow \mathbb{Z}^2$ the action of $f$ on $\pi_1(\mathbb{T}^d) \cong \mathbb{Z}^2. $ The linearization $A$ is given by a matrix of integer entries, moreover, it is known by \cite{AH} that $A$ induces a linear Anosov automorphism on $\mathbb{T}^d.$

From now on, in this work, we are ever considering Anosov endomorphisms $f$ such that its linearization $A$ is irreducible over $\mathbb{Q},$ meaning that, its characteristic polynomial is irreducible over $\mathbb{Q}[x].$ We also require that for $f$ the stable and unstable directions are non-trivial. The term endomorphism is ever used in the non-invertible setting.

\begin{theorem}\cite{MoTa}
An Anosov endomorphism $f:\mathbb{T}^n \rightarrow \mathbb{T}^n$ is conjugated with its linearization if and only if $f$ is special.
\end{theorem}

More recently, some results were obtained related to the existence of smooth conjugacy between a tori Anosov endomorphism and its linearization $A.$ See, for instance, \cite{micendo,CV,shi} . We highlight the result from \cite{shi}.
 \begin{theorem}\label{shi} Let $f: \mathbb{T}^n \rightarrow \mathbb{T}^n$  be an Anosov endomorphism such that $\dim E^s_f = 1$ and $A: \mathbb{T}^n \rightarrow \mathbb{T}^n$ its linearization.  Then $f$ is special, if and only if, $\lambda^s_f(p) = \lambda^s_A, \forall p \in Per(f).$ In this case the conjugacy between $f$ and $A$ is at least $C^{1 + \alpha}$ restricted to stable leaves.\end{theorem}

Relying on the results of \cite{shi},  we found a dichotomy involving the unstable Lyapunov exponent of a special Anosov endomorphism of the torus induced by the conjugacy with the linearization. Suppose that  $f: \mathbb{T}^2 \rightarrow  \mathbb{T}^2$ is an Anosov endomorphism and its linearization $A.$ Consider that $f$ and $A$ are conjugated by $h,$ such that $h \circ f = A\circ h.$ From now on we denote by $\widetilde{m} = h^{-1}_{\ast}(m),$ where $m$ is the usual volume form of $\mathbb{T}^2.$ Of course $\widetilde{m}$ is ergodic for $f$ and since $h$ is a homeomorphism $supp(\widetilde{m}) =  \mathbb{T}^2.$ We denote by $\lambda^u_f(\widetilde{m})$ the $\widetilde{m}-$typical unstable Lyapunov exponent of $f.$ We can prove the following.

\begin{maintheorem} \label{teo1} Let   $f: \mathbb{T}^2 \rightarrow  \mathbb{T}^2 $ be a smooth Anosov endomorphism with linearization $A,$ such that $f$ and $A$ are conjugated by $h,$ for which $h \circ f = A \circ h.$ Consider $ Z = \{x \in \mathbb{T}^2| \; \lambda^u_f(x) = \lambda^u_f(\widetilde{m})  \},$ then
either $Z$ meets every unstable leaf in a Lebesgue null set of the leaf, or $Z = \mathbb{T}^2$ and in this last case $f$ and $A$ are smoothly conjugated.
\end{maintheorem}


\begin{maintheorem}\label{teo2} Let $f: \mathbb{T}^2 \rightarrow \mathbb{T}^2 $ be a $C^{\infty}$ an Anosov endomorphism. Suppose that for any periodic points $p,q$ hold $\lambda^s_f(p) = \lambda^s_f(q)$ and $\lambda^u_f(p) = \lambda^u_f(q),$ then $f$ and its linearization $A$ are smoothly conjugated.
\end{maintheorem}

An immediate consequence of the previous Theorem is.

\begin{cor}\label{cor} Let $f: \mathbb{T}^2 \rightarrow \mathbb{T}^2 $ be a $C^{\infty}$ an Anosov endomorphism. Suppose that for any $x \in \mathbb{T}^2$ are defined the Lyapunov exponents $\lambda^s_f(x)$ and $\lambda^u_f(x),$ then $f$ and its linearization $A$ are smoothly conjugated.
\end{cor}

\begin{maintheorem}\label{teoplus} Let $f: \mathbb{T}^2 \rightarrow \mathbb{T}^2 $ be a $C^{\infty}$ special Anosov endomorphism, with degree $d \geq 2$ and $A$ its linearization. The following are equivalent.
\begin{enumerate}
\item $f$ preserves a measure $\mu$ absolutely continuous with respect to $m.$
\item $f$ is smoothly conjugated with its linearization $A.$
\item $f$ preserves a measure $\mu$ absolutely continuous with respect to $m,$ with $C^1$ density.
\item For any point $p $ such that $f^n(p) = p,$ for some integer $n \geq 1,$ holds $Jf^n(p) = d^n.$
\item There exists $c>0$ such that $Jf^n(p) = c^n,$ for any $p$ such that $f^n(p) = p,$ for some $n \geq 1$ is an integer number.
\end{enumerate}

\end{maintheorem}

Our next Theorem concerns on absolute continuity of foliations. More grossly a foliation $\mathcal{F}$ is absolutely continuous if one can decompose the Lebesgue measure on the absolute continuous measure on each leaf of the foliation and integrate it on each component as in Fubini's Theorem. For more about absolute continuity, we refer \cite{PVW}.
\begin{definition} Let $M$ be a compact and Riemannian manifold $M$ with a volume form $m.$ A foliation $\mathcal{F}$ of $M$ is absolutely continuous if for every $x \in M$ there is a local open neighborhod $U \subset M,$ with $x \in U$ satisfying: given a measurable set $Z \subset U, $ $m(Z)=0$ if and only if in $U$ there is a full set $F \subset U$ such that for evey $p \in F$ the component of leaf $\mathcal{F}(p) \cap U $ containing $p$ meets $Z$ on a zero Lebesgue measure set of the leaf.
\end{definition}

\begin{maintheorem}\label{teo3} Let $A:\mathbb{T}^3 \rightarrow \mathbb{T}^3 $ be a linear Anosov endomorphism which is irreducible over $\mathbb{Q}$ and it has three Lyapunov exponents $\lambda^s_A < 0  < \lambda^{wu}_A < \lambda^{su}_A.$ Then there is a $C^1$ neighborhood $\mathcal{U}$ of $A,$ such that for any $f \in \mathcal{U}$ is defined an expanding quasi-isometric foliation $\mathcal{F}^{wu}_f.$ Moreover, for a given $f \in \mathcal{U}$ being a $C^r, r\geq 2,$  special and $m-$preserving Anosov endomorphism holds $\mathcal{F}^{wu}_f$ is absolutely continuous, if and only if, $\lambda^{wu}_f(p) = \lambda^{wu}_A,$ for any $p \in Per(f).$
\end{maintheorem}

\begin{cor} Let $A:\mathbb{T}^3 \rightarrow \mathbb{T}^3 $ be a linear Anosov endomorphism which is irreducible over $\mathbb{Q}$ and it has three Lyapunov exponents $\lambda^s_A < 0  < \lambda^{wu}_A < \lambda^{su}_A.$ Then there is a $C^1$ neighborhood $\mathcal{U}$ of $A$ such that for a given $f \in \mathcal{U}$ being a $C^r, r\geq 2,$ special and $m-$preserving Anosov endomorphism the following are equivalent.
\begin{enumerate}
\item $\mathcal{F}^{wu}_f$ is absolutely continuous.
\item $\lambda^{wu}_f(p) = \lambda^{wu}_A,$ for any $p \in Per(f).$
\item $\lambda^{su}_f(p) = \lambda^{su}_A,$ for any $p \in Per(f).$
\item $\lambda^{\sigma}_f(p) = \lambda^{\sigma}_A, \sigma \in \{s,wu,su\},$ for any $p \in Per(f).$
\end{enumerate}
\end{cor}

\section{Preliminaries on SRB measures for endomorphisms} The reader more acquainted with SRB theory can jump straight to the proofs of the theorems.

 At this moment we need to work with the concept of SRB measures for endomorphisms. In fact, SRB measures play an important role in the ergodic theory of differentiable dynamical systems. For $C^{1+\alpha}-$systems these measures can be characterized as ones that realize the Pesin Formula or equivalently the measures for which the conditional measures are absolutely continuous w.r.t. Lebesgue restricted to local stable/unstable manifolds. We go to focus our attention on the endomorphism case. Before proceeding with the proof let us give important and useful definitions and results concerning SRB measures for endomorphisms.

First, let us recall an important result.

\begin{theorem}[\cite{QXZ}] Let $(M,d)$ be a compact metric space and $f: M \rightarrow M$ a continuous map. If $\mu$ is an $f-$invariant Borelian probability measure, the exist a unique $\tilde{f}-$invariant Borelian probability measure $\tilde{\mu}$ on $M^f,$ such that $\mu(B) = \hat{\mu}(p^{-1}(B)).$
\end{theorem}

\begin{definition}A measurable partition $\eta$ of $M^f$ is said to be subordinate to
$W^u-$manifolds of a system $(f, \mu)$ if for $\hat{\mu}$-a.e. $\tilde{x} \in M^f,$ the atom $\eta(\tilde{x}),$ containing $\tilde{x},$ has the following properties:
\begin{enumerate}
\item  $p| \eta(\tilde{x}) \rightarrow p(\eta(\tilde{x}))$ is bijective;
\item  There exists a $k(\tilde{x})-$dimensional $C^1-$embedded submanifold $W(\tilde{x})$ of $M$ such that $W(\tilde{x}) \subset W^u(\tilde{x}),$
$$p(\eta(\tilde{x})) \subset W(\tilde{x})$$
and $p(\eta(\tilde{x}))$ contains an open neighborhood of $x_0 $ in $W(\tilde{x}).$ This neighborhood
being taken in the topology of $W(\tilde{x})$ as a submanifold of $M.$
\end{enumerate}
\end{definition}

We observe that by Proposition 3.2 of \cite{QZ}, every there exist   such partition for Anosov endomorphism and they are can be taken increasing, that means $\eta$ refines $\tilde{f}(\eta) .$ Particularly $ p(\eta(\tilde{f}(\tilde{x}))) \subset p(\tilde{f}(\eta(\tilde{x}))) .$

\begin{definition} Let $f: M \rightarrow M$ be a $C^2-$endomorphism preserving an invariant Borelian probability $\nu.$
We say that $\nu$ has SRB property if for every measurable partition $\eta $ of $M^f$ subordinate
to $W^u-$manifolds of $f$  with respect to $\nu$, we have $p(\hat{\nu}_{\eta{(\tilde{x})}})  \ll m^u_{p(\eta(\tilde{x}))},$  for $\hat{\nu}-$a.e. $\tilde{x}$, where
$\{\hat{\nu}_{\eta{(\tilde{x})}} \}_{\tilde{x} \in M^f}$
is a canonical system of conditional measures of $\hat{\nu}$ associated with $\eta,$
and $m^u_{p(\eta(\tilde{x}))} $ is the Lebesgue measure on $W(\tilde{x})$ induced by its inherited Riemannian metric as a submanifold of $M.$
\end{definition}

 In the case of above definition, if we denote by $\rho^u_f$ the densities of conditional measures $\hat{\nu}_{\eta({\tilde{x}})},$ the following relation holds
 \begin{equation} \label{conditionalU}
 \rho^u_f(\tilde{y}) =
\frac{\Delta^u_f(\tilde{x}, \tilde{y} )}{L(\tilde{x})},
 \end{equation}
for each $\tilde{y} \in \eta({\tilde{x}}),$ where
$$ \Delta^u_f(\tilde{x},\tilde{y}) = \displaystyle\prod_{k=1}^{\infty} \frac{J^uf(x_{-k})}{J^uf(y_{-k})}, \tilde{x} = (x_k)_{k \in \mathbb{Z}}, \tilde{y} = (y_k)_{k \in \mathbb{Z}}  $$
and
$$L(\tilde{x}) = \int_{\eta(\tilde{x})} \Delta^u_f(\tilde{x}, \tilde{y}) d \hat{m}^u_{\eta({\tilde{x}})}(\tilde{y}).$$

The measure $\hat{m}^u_{\eta({\tilde{x}})}$ is such that $p(\hat{m}^u_{\eta({\tilde{x}})})(B) = m^u_{p(\eta({\tilde{x}}))}(B).$  Therefore $$p(\hat{\nu}_{\eta({\tilde{x}})}) \ll m^u_{p(\eta({\tilde{x}}))},  $$ and
$$\rho^u_f(y) =
\frac{\Delta^u_f(\tilde{x}, \tilde{y})}{L(\tilde{x})}, y \in p(\eta({\tilde{x}})).$$

\begin{theorem}{\cite{PDL1}}\label{pesin1} Let $f : M \rightarrow M$ be a $C^2$ endomorphism and $\mu$ an
$f-$invariant Borel probability measure on $M.$ If $\mu \ll m,$ then there holds
Pesin's formula
\begin{equation}
h_{\mu}(f) =  \displaystyle\int_M \displaystyle\sum \lambda^i(x)^{+}m_i(x) d\mu.
\end{equation}
\end{theorem}

\begin{theorem}[\cite{QZ}] \label{pesin2} Let $f$ be a $C^2$ endomorphism on $M$  with an invariant Borel probability
measure $\mu$ such that $\log(|Jf(x)|) \in L^1(M,\mu).$ Then the entropy
formula
\begin{equation}\label{PesinU}
h_{\mu}(f) =  \displaystyle\int_M \displaystyle\sum \lambda^i(x)^{+}m_i(x) d\mu
\end{equation}
holds if and only if $\mu$ has SRB property.
\end{theorem}

For $C^2$ attractors there is a unique measure satisfying the Pesin entropy formula.

\begin{theorem}[Corollary 1.1.2 of \cite{QZ}]
Let $\Lambda$ be an Axiom A attractor of $f \in C^2(O, M)$  and assume
that $T_xf$ is nondegenerate for every $x \in \Lambda.$ Then there exists a unique $f-$invariant
Borel probability measure $\mu$ on $\Lambda$ which is characterized by each of the following
properties: \begin{enumerate}
\item $\mu$ has the SRB property.
\item The system $f : (\Lambda, \mu) \rightarrow (M, \mu)$ satisfies the entropy formula.
\item When $\varepsilon > 0$ is small enough, $\frac{1}{n} \displaystyle\sum_{i=0}^{n-1} \sigma_{f^k(x)} $  converges to $\mu$ as $n  \rightarrow +\infty$ for
Lebesgue almost every $x \in B_{\varepsilon}(\Lambda) = \{ y \in M|\; d(y, \Lambda) < \varepsilon\}.$
\end{enumerate}
\end{theorem}

The above result can be applied to $C^2$ transitive Anosov endomorphisms, particularly it holds for all Anosov endomorphisms of the torus, see \cite{AH}.

There are analogous formulations concerning subordinate partition with respect to stable manifolds, which can be take decreasing, that means $f^{-1} (\eta) \preceq \eta,$ see \cite{PDL2}, Proposition 4.1.1.   In the sense of hyperbolic repellers, including  Anosov endomorphisms, there is an important result concerning inverse SRB measures.

\begin{theorem}\label{teonegativo}[Theorem 3 of \cite{Mh1} and Theorems 2.3 and 2.6 of \cite{PDL2}] Let $\Lambda$ be a connected hyperbolic repellor for a smooth $f: M \rightarrow M .$ Assume that $f$ is $d$ to one, then there is a unique $f-$invariant probability measure $\mu^{-}$ on $\Lambda$ satisfying the inverse Pesin formula

\begin{equation}\label{PesinS}
h_{\mu^{-}}(f) = \log(d) - \displaystyle\int_M \displaystyle\sum \lambda^i(x)^{-}m_i(x) d\mu^{-}.
\end{equation}
In addition, the measure $\mu^{-}$ is characterized by having absolutely continuous conditional measures on local stable
manifolds.

\end{theorem}

In the setting of the previous Theorem, if $(f, \mu)$ satisfies the Stable Pesin Formula \ref{PesinS}, then for a given subordinate partition $\eta,$ with respect to stable manifolds, we have $$ \mu_{\eta (x)} \ll m^s_{\eta (x)},$$ for $\mu-$ a.e $x \in M.$ Moreover
\begin{equation}\label{conditionalS}
\rho^s_f(x) = \frac{\Delta^s_f(x,y)}{\int_{\eta(x)} \Delta^s_f(x,y)dm^s_{\eta (x)} }, \; \forall y \in \eta(x).
\end{equation}

Here $\Delta^s_f(x,y) = \prod_{k = 0}^{\infty} \frac{Jf(f^k(x))}{Jf(f^k(y))}\cdot \frac{J^sf(f^k(x))}{J^sf(f^k(y))}.$
See \cite{PDL2} as a reference.

The theorems on Pesin formulas are true in our setting since every tori Anosov endomorphism is transitive, see \cite{AH}.

We finalize the preliminaries section with a lemma whose proof is essentially the same as Corollary 4.4 of \cite{Llave92}, up to minor adjustments using local inverses.

\begin{lemma} For a $C^{k}, k \geq 2,$ Anosov endomorphism, the conditional measures of stable and unstable SRB measures restricted to stable and unstable leaves respectively are $C^{k-1}.$ In particular, if $f$ is smooth, then the conditional measures are smooth.
\end{lemma}

\section{Proof of Theorem \ref{teo1}}

The proof of Theorem \ref{teo1} deals with techniques from \cite{MT13, MT, micendo}. For completeness, we redo here similar arguments used in past works. We denote by $\overline{f}: \mathbb{R}^2 \rightarrow \mathbb{R}^2 $ the lift of $f$ and for simplicity we denote also by $A: \mathbb{R}^2 \rightarrow \mathbb{R}^2 $ the lift of $A: \mathbb{T}^2 \rightarrow \mathbb{T}^2. $

\begin{definition} \label{quasi isometric}
A foliation $W$ of $\mathbb{R}^n$  is quasi-isometric if there exist positive constants $Q$ and $b$
such that for all $x, y$ in a common leaf of $W$ holds
$$d_W(x, y) \leq Q^{-1} || x - y|| + b.$$
Here $d_W$ denotes the Riemannian metric on $W$ and $\|x-y\|$ is the Euclidean distance.
\end{definition}

\begin{remark}\label{remarkquasi} Observe that if $||x - y||$ is large enough, we can  consider $b = 0, $ in the above definition.
\end{remark}

\begin{lemma} Let $f: \mathbb{T}^2 \rightarrow \mathbb{T}^2$ be an Anosov endomorphism.  Then  the lifted unstable and stable foliations $\mathcal{F}^u_{\overline{f}}$ and $\mathcal{F}^s_{\overline{f}}$ are quasi isometric
\end{lemma}
\begin{proof}
First, we observe that $h$ sends unstable manifolds of $A$ in unstable manifolds of $f.$ Moreover, considering any lift $H$ of $h$ hods that $||H^{-1} - id||_{\infty} < R$ for some positive constant $R.$ So if $L$ is a unstable leaf of $A$ and $W = H^{-1}(L)$ an unstable leaf of $\overline{f},$ then $W \subset B_R(L),$ where $B_R(L) = \{x \in \mathbb{R}^2| \;  d(x, L) < R\}.$ From \cite{HH}, the foliations $\mathcal{F}^u_{\overline{f}}$ and $\mathcal{F}^s_{\overline{f}}$ have a global product structure, meaning that, if $F^s \in \mathcal{F}^s_{\overline{f}} $ and $^u \in \mathcal{F}^u_{\overline{f}},$ then $F^s \cap F^u$ consists of a single point. With these two properties it is possible to prove that  $\mathcal{F}^u_{\overline{f}}$ and $\mathcal{F}^s_{\overline{f}}$ are quasi-isometric foliations, see the detailed argument in \cite{H}. The reader can see nice properties and Proposition 2.16 of \cite{H} to complete the proof. The dictionary from \cite{H} to this proof is the following. Here a unstable and stable leaves $F^u$ and $F^s$  play like $W^c_f$  and $W^{us}_f$ in \cite{H}.
\end{proof}

\begin{cor}\label{nice}
 For any Anosov endomorphism $f: \mathbb{T}^2 \rightarrow \mathbb{T}^2$ with linearization $A,$ the following properties hold in the universal covering:
\begin{enumerate}
\item For each $k \in \mathbb{N}$ and $C > 1$ there is $M$ such that,
$$||x -  y|| > M \Rightarrow \frac{1}{C} \leq \displaystyle\frac{||\overline{f}^kx - \overline{f}^k y|| }{||A^kx - A^ky||} \leq C.$$
\item  $ \displaystyle\lim_{||y - x || \rightarrow +\infty} \frac{y-x}{||y - x ||} = E_A^{\sigma}, \;\;  y \in \mathcal{F}^{\sigma}_{\overline{f}} (x), \sigma  \in \{s,  u\},$
uniformly.
\end{enumerate}
\end{cor}

\begin{proof}
The proof is in the lines of \cite{H} and we repeat for completeness. Let $K$ be a fundamental domain of $\mathbb{T}^2$ in $\mathbb{R}^2.$ Restricted to $K $ it is true that

$$||\overline{f}^k -  A^k|| < +\infty, $$
for $\overline{x} \in \mathbb{R}^2,$ there are $x \in K$ and $\overrightarrow{n} \in \mathbb{Z}^2,$ such that $\overline{x}  = x + \overrightarrow{n}, $ since $f_{\ast} = A,$ we obtain:

$$ ||\overline{f}^k(\overline{x}) - A^k(\overline{x})|| = ||\overline{f}^k(x + \overrightarrow{n}) - A^k(x +\overrightarrow{ n})|| = ||\overline{f}^k(x) +A^k \overrightarrow{n} - A^kx - A^k\overrightarrow{n}  || < +\infty. $$

Now, for every $x, y \in \mathbb{R}^2,$

$$||\overline{f}^k x - \overline{f}^ky|| \leq ||A^k x - A^k y|| + 2||\overline{f}^k - A^k||_0$$

$$||A^k x - A^ky|| \leq ||\overline{f}^k x - \overline{f}^k y|| + 2||\overline{f}^k - A^k||_0,$$
where,

 $$||\overline{f}^k - A^k||_0 = \max_{x \in K}\{||\overline{f}^k(x) - A^k(x)||\}.$$

Since $A$ is non-singular,  if $||x -  y|| \rightarrow + \infty,$ then $||A^kx - A^k y|| \rightarrow + \infty.$

So dividing both expressions by $||A^kx - A^k y|| $ and doing $||x -  y|| \rightarrow + \infty$ we obtain the proof of the first item.

For the second item, we just consider the case of stable sub bundles. For unstable sub bundles, we  consider $A^{-1}$ and $(\overline{f})^{-1}$  and same proof holds.

Let $|\theta^s| = \max\{\,|\theta| \;| \; \theta\; \mbox{is eigenvalue of $A$ and}\; 0 < |\theta| < 1 \}.$  Fix a small $\varepsilon > 0$ and consider $\delta > 0,$ such that $0 < (1+ 2\delta)|\theta^s| < 1.$ If $f$ is sufficiently $C^1-$close to $A,$ then $\overline{f}$  is an Anosov diffeomorphism  on $\mathbb{R}^2$ with contracting constant less than $(1+ \delta)|\theta^s|.$

Using the hyperbolic splitting, there is $k_0 \in \mathbb{N},$ such that if $v \in \mathbb{R}^2,$  $k > k_0$ and

$$||A^k v || < (1 + 2\delta)^k |\theta^s|^k ||v||,$$
then

$$||\pi^u_A(v)|| < \varepsilon ||\pi^s_A(v)||.$$

Consider $k > k_0$ and $M$ sufficiently large, satisfying  the first item  with $C = 2$ and in accordance with remark \ref{remarkquasi}.

Take $y \in \mathcal{F}^s_{\overline{f}}(x)$ and $||x - y|| > M.$ Let $d^s$ to denote the Riemannian distance on stable leaves of $\mathcal{F}^s_{\overline{f}},$ using quasi isometry property of the foliation $\mathcal{F}^s_{\overline{f}},$ we get

$$d^s(\overline{f}^k x, \overline{f}^k y) <  ((1 + \delta)|\theta^s|)^k d^s(x,y) \Rightarrow $$
$$ ||\overline{f}^k x - \overline{f}^k y|| <  ((1 + \delta)|\theta^s|)^k (Q^ {-1}|| x - y||) \Rightarrow $$
$$ ||A^k x - A^k y|| < 2 ((1 + \delta)|\theta^s|)^k (Q^ {-1}|| x - y||).$$

Finally,  for large $k$ we obtain:

$$2Q^{-1} ((1 + \delta)|\theta^s|)^k  \leq ((1 + 2\delta) |\theta^s|)^k,$$
So,

$$||\pi^u_A(x - y)|| < \varepsilon ||\pi^s_A(x- y)||.$$

\end{proof}

\begin{lemma} \cite{MT} \label{linalg}
 Let $f : \mathbb{T}^2 \rightarrow \mathbb{T}^2$ be an Anosov endomorphism with linearization  $A: \mathbb{T}^2 \rightarrow \mathbb{T}^2,$ such that $dim E^u_A = 1.$ Then   for all $n \in \mathbb{N}$ and $\varepsilon > 0$ there exists $M$ such that for $x, y$ with $y \in \mathcal{F}^u_{\overline{f}}(x)$ and $||x -  y||> M$ then
 $$
   (1 - \varepsilon)e^{n\lambda^{u}_A } ||y -x|| \leq \|A^n(x) - A^n(y)\| \leq (1 + \varepsilon)e^{n\lambda^{u}_A } ||y -x||
 $$
where $\lambda^{u}$ is the Lyapunov exponent of $A$ corresponding to $E^{u}_A.$
 \end{lemma}

\begin{proof} Denote by $E^{u}_A$ the eigenspace corresponding to $\lambda^{u}_A$ and $|\mu| = e^{\lambda^{u}_A},$ where $\mu$ is the eigenvalue of $A$ in the $E^{u}_A$ direction.

Let  $N \in \mathbb{N}$ and choose $x, y \in \mathcal{F}^{u}_f(x),$ such that $|| x - y || > M.$ By corollary \ref{nice},we obtain

$$ \frac{x - y}{|| x - y||} = v + e_M,$$
where the vector $v = v_{E^{u}_A}$ is a unitary eigenvector of  $A,$ in the  $E^{u}_A$ direction and  $e_M$ is a correction vector that converges to zero uniformly as  $M$ goes to  infinity. It leads us to

$$A^N \left( \frac{x - y}{|| x - y||} \right) = \mu^N v + A^N e_M = \mu^N \left(\frac{x - y}{|| x - y||} \right) -\mu^N e_M  + A^N e_M  $$
it implies that

\begin{align*} || x - y || (|\mu|^N - |\mu|^N ||e_M|| - ||A||^N || e_M||) \leq   || A^N (x - y)|| \\ \leq || x - y || (|\mu|^N + |\mu|^N ||e_M|| + ||A||^N || e_M||).
\end{align*}

Since $N$ is fixed, we can choose  $M > 0,$ such that

$$ |\mu|^N ||e_M|| + ||A||^N || e_M|| \leq \varepsilon |\mu|^N.$$
and the lemma is proved. \end{proof}


\begin{lemma} Let $f $ be a special Anosov endomorphism  and $A$ its linearization such that $A$ is irreducible over $\mathbb{Q}$ with $\dim E^u_A = \dim E^u_f = 1. $ If there is a unstable leaf meeting $Z$ in a positive volume set of the leaf, then $\lambda^u_f(\widetilde{m}) =  \lambda^u_A.$
\end{lemma}

\begin{proof}
 Let us to prove first $\lambda^u_f(\widetilde{m}) \leq \lambda^u_A.$ Suppose that  $\lambda^u_f(\widetilde{m}) >  \lambda^u_A ,$ so we can choose $\varepsilon > 0$ such that    $\lambda^u_{f}(\widetilde{m}) > (1 + 5 \varepsilon) \lambda^u_A,$ for a small $\varepsilon > 0.$

\begin{equation}
 m^u_x(\mathcal{F}^u_{\overline{f}}(x) \cap Z) > 0
\label{1}
\end{equation}
where $m^u_x$ is the Lebesgue measure of the leaf $\mathcal{F}^u_{\overline{f}}(x)$.
Consider an  interval $[x,y]_u \subset \mathcal{F}^u_{\overline{f}}(p) $ satisfying
 $m^u_p([x,y]_u \cap Z) > 0$ such that the length of $[x,y]_u$ is bigger than $M$ as required in the lemma \ref{linalg} and corollary \ref{nice}. We can choose $M$ such that

$$||Ax - Ay|| < (1 + \varepsilon)e^{\lambda^u_A } ||y -x|| $$
and

$$\frac{|| \overline{f}(x) - \overline{f}(y)|| }{ ||Ax - Ay||} < 1 + \varepsilon. $$
whenever
 $d^u(x, y) \geq M,$ where $d^u$ denotes the Riemannian distance in unstable leaves. The above equation implies that $$ || \overline{f}(x) -\overline{f}(y)|| < (1+ \varepsilon)^2 e^{\lambda^u_A} || y - x||.$$

Inductively, we assume that for  $n \geq 1,$ then

\begin{equation}
|| \overline{f}^n(x) - \overline{f}^n(y)||< (1+\varepsilon)^{2n} e^{n \lambda^u_A }|| y - x||. \label{induction}
\end{equation}

Since $f$ expands uniformly  on the $u-$direction we obtain $d^u(\overline{f}^n(x), \overline{f}^n(y)) > M,$ consequently

\begin{eqnarray*}
||\overline{ f}(\overline{f}^nx) - \overline{f}(\overline{f}^ny)|| &<& (1+\varepsilon)|| A(\overline{f}^nx) - A(\overline{f}^ny)|| \\  &<& (1 + \varepsilon)^2 e^{\lambda^u_A} || \overline{f}^nx - \overline{f}^n y||\\ &<&
 (1+\varepsilon)^{2(n+1)} e^{(n+1)\lambda^u_A}.
\end{eqnarray*}

For each $n > 0,$ let $Z_n \subset Z$ be the following set

$$Z_n = \{ x \in Z \colon\;\; \|D\overline{f}^k(x)|E^u_{\overline{f}}(x) \| > (1+2\varepsilon)^{2k} e^{k\lambda^u_A} \;\; \mbox{for any} \;\; k \geq n\}. $$
At this point $m(Z) > 0$ and $Z_n := (Z_n \cap Z) \uparrow Z,$ as $(1 + 5 \varepsilon) > (1 + 2\varepsilon)^2,$ for small $\varepsilon > 0.$

Define the number $\alpha_0 > 0$ such that:

$$\displaystyle\frac{m^u_p([x,y]_u \cap Z)}{m^u_p([x,y]_u)} = 2 \alpha_0.$$

Since $Z_n \cap [x,y]_u \uparrow Z \cap [x,y]_u, $ there is $n_0 \in \mathbb{N} ,$ such that $n \geq n_0,$ then

 $$ m^u_p ([x,y]_u \cap Z_n) = \alpha_n \cdot m^u_p([x,y]_u),$$
for $\alpha_n > \alpha_0.$

Thus, for $n \geq n_0:$

\begin{eqnarray}
||\overline{f}^nx - \overline{f}^ny || &>&  Q \displaystyle\int_{[x,y]_u \cap Z_n} ||Df^n(z)|| dm^u_p(z) >  \\ &>&
 Q (1+ 2\varepsilon)^{2n}
e^{n \lambda_A^u } m^u_p ([x,y]_u \cap Z_n) \\ &>& \alpha_0 Q^2 (1 + 2\varepsilon)^{2n} e^{n\lambda^u_A} \|x-y\|. \label{conclusion}
\end{eqnarray}
The inequalities $(\ref{induction})$ and $(\ref{conclusion})$ give a contradiction. We conclude $\lambda^u_f(\widetilde{m}) \leq \lambda^u_A.$

It remains to prove that $ \lambda^u_f(\widetilde{m}) \geq \lambda^u_A. $ In fact
$$\lambda^u_A = h_{m}(A) = h_{\widetilde{m}}(f) \leq \lambda^u_f(\widetilde{m}), $$
the last inequality follows from Ruelle's inequality. Particularly $ h_{\widetilde{m}}(f) = \lambda^u_f(\widetilde{m}),$ so $\widetilde{m}$ satisfies the Pesin formula, that means that it is a SRB measure for $f.$
\end{proof}

\begin{lemma}\label{lemasmooth} Let $f$ and $A$ be as Theorem \ref{teo1} and suppose that $\lambda^u_f(\widetilde{m}) = \lambda^u_A.$ Then $f$ and $A$ are smoothly conjugated.
\end{lemma}

\begin{proof}
 Consider $V$ a small local open neighborhood of $\mathbb{T}^2$ foliated by  $\mathcal{F}^u_f.$  First we note that $\widetilde{m}$ is ergodic and since $h$ is continuous, $\widetilde{m}$ is supported on $\mathbb{T}^2.$ So $\widetilde{m}(U) > 0$ for any non empty open set $U.$ So by ergodicity, $\widetilde{m}$ almost every where $x$ has positive dense orbit.

We can choose an $f-$orbit $\tilde{x} = (x_n)_{n \in \mathbb{Z}}$ is such that $(x_n)_{n \geq 0} $ is dense and  there are atoms $\eta_k(\tilde{f}^k(\tilde{x})), k \geq 0,$ for which are defined the $C^{\infty}-$conditional measure $\rho^u_f,$ for  subordinate partitions w.r.t unstable leaves. In fact, for the partition $\eta_k = \tilde{f}^k(\eta),$ where $\eta$ is any subordinate partition w.r.t unstable leaves, let $\mu$ be the unique measure such that $p_{\ast}\mu = \widetilde{m}.$ Consider the $\mu-$full measure set of points $X_k,$ of points satisfying $(\ref{conditionalU}).$ Now take $X = \bigcap_{k=0}^{+\infty} X_k,$ and finally $\mathcal{T} = \bigcap_{j=0}^{+\infty}\tilde{f}^{-j} X .$ The projection on $\mathbb{T}^2$ of $\mathcal{T}$ has $\widetilde{m}-$full measure, by ergodicity and knowing $supp(\widetilde{m}) = \mathbb{T}^2,$ we can choose the orbit.

Since $h_{\ast}(\widetilde{m}) = m,$ then  $h$ sends conditional measures of $(f, \widetilde{m})$ in conditional measures of $(A, m).$ Since these measures are equivalent to Riemannian measures of unstable leaves, so $h$ sends null sets of $p(\eta(\tilde{x}))$ in null sets of $p(\eta(\tilde{h}(\tilde{x})))$ with respect to  Riemannian measures of unstable leaves, where $\tilde{h}$ is the conjugacy at level of limit inverse space between $\tilde{f}$ and $\tilde{A}.$

Consider $B^u_{x_0} \subset \eta(\tilde{x}) $ a small open unstable arc. Since $h$ is absolutely continuous
$$\int_{B^u_{x_0}} \rho^u_f(y) dy  = \int_{h({B^u_{x_0}})} \rho^u_A(y) dy = \int_{B^u_{x_0}} \rho^u_A(h(y)) h'(y) dy,  $$
therefore solving the O.D.E.
\begin{equation}\label{ODE}
x' = \frac{\rho^u_f(t)}{\rho^u_A(x)}, x(x_0) =  h(x_0) ,
\end{equation}
We conclude that $h$ is $C^{\infty}$ on $B^u_{x_0}.$

Given $z_0 \in V,$ since $\{f^n(x_0), n \geq 0\}$ is dense in $\mathbb{T}^2,$ there is a sequence of iterated of $f^{n_k}(x_0), n_k \geq 0$ such that $f^{n_k}(x_0) \rightarrow z_0.$ Since $x_0$ lies in the interior of $B^u_{x_0},$ by Theorem 1.12 of \cite{PRZ}, up to take a subsequence we can suppose that the sequence connected components of the arcs  $W_n \subset f^n(B^u_{x_0}) \cap V$ containing $f^n(x_0)$ is such that  $W_n \rightarrow_{C^1} W_{z_0},$ where $\mathcal{F}^u_f(z_0)$ is the connected component of $\mathcal{F}^u_f(z_0) \cap V$ containing $z_0.$

Since the subordinate partition can be taken increasing, see Proposition 3.2 of \cite{QZ}, the conjugacy $h$ restricted to $W_n$ satisfies an analogous O.D.E, as in $(\ref{ODE}).$

Normalizing the conditional measures such that $$ \int_{W_n} c_n\cdot \rho^u_f(t) dVol_{W_n} = 1,$$ since $h_{\ast}(\rho^u_f(t) dVol_{W_n}) = \rho^u_A(t) dVol_{h(W_n)},$ then $h$ send normalized conditional measures into normalized conditional measures. To avoid carrying $c_n,$ since the same constant works for $f$ and $A,$ for simplicity we consider $c_n=1.$

For any $y \in W_n,$ take the initial condition $y_0 = f^n(x_0),$ we get $$\rho^u_f(y) = \alpha_n \cdot\Delta^u(y_0, y),$$ for some constant $\alpha_n.$ Since $V$ can be taken with compact closure, the sequence of $|\alpha_n|$ is bounded and far from zero. For $A$ there are corresponding constants $\beta_n$ with same properties, so $\frac{|\alpha_n|}{|\beta_n|}$ is also bounded and far from zero.  For simplicity of writing let us consider constant $\alpha_n = \beta_n.$

In this way, by relation \eqref{conditionalU},  $h$ satisfies the following O.D.E,
$$x' = \frac{\Delta^u_f( y_0 , t)}{\Delta^u_A( h(y_0), x)}, x(y_0) =  h(y_0), \Leftrightarrow x' = \Delta^u_f(t), x(y_0) =  h(y_0). $$
for each pair of connected component $W_n$ and  $h(W_n).$

Denoting by $h_n$ the solution of the above equation, we note that the solution $h_n$ is smooth. The map $h_n$ is the restriction of the conjugacy $h$ on $W_n.$ Analogous to Lemma 4.3 of \cite{Llave92}, for each component $W_n$ we have a collection $\{h_n: W_n \rightarrow h(W_n)\}_{n = 1}^{\infty},$  is uniform bounded as well the collection of their derivatives of order $r = 1,2,\ldots.$   By an Arzela-Ascoli argument type applied to a sequence $h_n$ and the sequence of their derivatives, we conclude that $h$ is $C^{\infty}$ restricted to $\mathcal{F}^u_f(z_0),$ if $x_n \rightarrow z_0.$

We conclude that $h$ is smooth on the leaves of $\mathcal{F}^u_f.$ By Theorem \ref{shi}, since $f $ is special, also $h$ is $C^{1+\alpha}$ on the leaves of $\mathcal{F}^s_f.$ So $f$ and $A$ have same periodic data, from Theorem B of \cite{micendo}, $f$ and $A$ are smoothly conjugated, consequently $Z = \mathbb{T}^2.$

\end{proof}

\section{Proof of Theorem \ref{teo2}}

\begin{proof}
For any periodic points $p$ and $q$ of $f$ holds
$$ \lambda^s_f(p) = \lambda^s_f(q), \lambda^u_f(p) = \lambda^u_f(q),$$

By Theorem 5.1 of \cite{shi}, the map $f$ is a special Anosov endomorphism. Since  $\lambda^u_f(p) = \lambda^u_f(q) = \lambda^u_f$ for any $p,q \in Per(f),$ by Livsic's Theorem  $\lambda^u_f(x) = \lambda^u_f.$ Finally Theorem \ref{teo1} implies that $f$ and $A$ are smoothly conjugated.
\end{proof}

To prove the Corollary \ref{cor} we will use the specification property. Since $f$ has specification property (see \cite{SY}), then for any periodic points $p$ and $q$ of $f$
$$ \lambda^s_f(p) = \lambda^s_f(q), \lambda^u_f(p) = \lambda^u_f(q),$$
for a proof of these identities, we refer to \cite{micendo}.

By Theorem \ref{teo2} the endomorphisms $f$ and $A$ are smoothly conjugated.

\section{Proof of Theorem \ref{teoplus} }

\begin{lemma}\label{lema2} Let $f: M  \rightarrow M$ be a transitive Anosov endomorphism with degree $k \geq 1$ where $M,$ is  a $C^{\infty}$ compact and connect Riemannian manifold. Then $f$ preserves a $C^1$ volume form $m$ on $M$ if and only if $Jf^n(p) = k^n,$ for any $p \in M$ such that $f^n(p) = p,$ with $n \geq 1.$
\end{lemma}

\begin{proof}
By simplicity consider the case of $f(p) = p.$ Let $V$ be a small neighborhood of $p$ and $V_1, \ldots, V_k$ mutually disjoint neighborhoods such that $f: V_i \rightarrow V$ is a diffeomorphism with inverse $g_i: V \rightarrow V_i.$ For each $i,$ applying the   chain's rule to $g_i \circ f$ we get $g_i'(p) = [f'(p)]^{-1},$ and thus $Jg(p) = \frac{1}{Jf(p)}.$ Suppose that $Jf(p) > k,$ so up to adjust $V,$ we can suppose that $Jf(x) > k$ for any $x \in V,$ it leads us to $Jg_i(z) <\frac{1}{k}$ for any $z \in V.$
So

$$m(V_i) = m(g_i(V)) = \int_{V} Jg_i(z) dm < \frac{1}{k} m(V), i =1, \ldots,k.$$
and it implies
$$ \sum_{i=1}^km(V_i) < m(V) \Leftrightarrow m(f^{-1} (V)) < m(V),$$
and thus $f$ does not preserve $m.$ Analogously $Jf(p)$ can not be bigger than $k.$ If $f^n(p) = p,$ for some $n > 1,$ then we apply the same idea replacing $f$ by $f^n$ with degree $k^n.$ Now $p$ is a fixed point of $f^n$ and everything as before works in this setting.

Now suppose that $Jf^n(p) = k^n,$ for any $p \in M$ such that $f^n(p) = p.$ By Livsic's Theorem there is a $C^1$ function $\phi:M \rightarrow \mathbb{R}$ such that

\begin{equation}\label{density}
\log(J f(x)) - \log(k) =  \phi( f(x)) - \phi(x).
\end{equation}

It leads us to  $$Jf(x) e^{-\phi(f(x))} = ke^{-\phi(x)}. $$

Define the measure $dm = e^{-\phi(x)}dw,$ where $w$ is a volume form defined on $M.$ Consider $V$ be a small nonempty open ball and $V_1, V_2, \ldots, V_k$ being mutually disjoint pre images, $f(V_i) = V.$

$$m(V) = \nu(f(V_i)) = \int_{f(V_i)} e^{-\phi(y)}dw  =  \int_{V_i} Jf(x) e^{-\phi(f(x))}dw = \int_{V_i} k e^{-\phi(x)}dw = k m(V_i) $$
$$m(V_i) = \frac{1}{k} m(V)$$

$$ m(V)  = \sum_{i = 1}^k m(V_i) =  m(f^{-1}(V) ).$$
We conclude that $f$ preserves the volume form $m.$ Particularly we just proved $(3) \Leftrightarrow (4)$ of Theorem \ref{teoplus}.

\end{proof}

Since the previous Lemma provide us $(3) \Leftrightarrow (4)$, the strategy  finish the proof will be $(1) \Rightarrow (2) \Rightarrow (3) \Rightarrow (1)$ and after $(4) \Leftrightarrow (5).$

Suppose that $f$ preserves $\mu \preceq m.$ Since $\mu \preceq m$ and $\mathcal{F}^s_f$ is an absolutely continuous foliations,
$$h_{\mu}(f) = \log(d) - \lambda^s_f = \log(d) - \lambda^s_A = h_m(A).$$
Note that, since also $h_{\widetilde{m}}(f) = h_m(A) = \log(d) - \lambda^s_f,$ we get $\widetilde{m} = \mu$ by uniqueness of inverse SRB of $f.$

By Pesin's formula too
$$h_{\widetilde{m}}(f) = \int_{\mathbb{T}^2} \lambda^u_f(x)d\widetilde{m}(x) = \lambda^u_f(\widetilde{m}).$$
Since $\widetilde{m}$ is ergodic and an absolutely continuous measure and  $\mathcal{F}_f^u$ is absolutely continuous foliation, then there is exist a leaf $\mathcal{F}^u_f$ such that intersects the set $Z,$ as in Theorem \ref{teo1}, in a positive measure of the leaf. So by Theorem \ref{teo1}, $f$ and $A$ are smoothly conjugated. We proved $(1)\Rightarrow (2).$

Now the sequence $(2) \Rightarrow (3) \Rightarrow (1) $ is straightforward, as well $(4) \Rightarrow (5).$ Since Lemma \ref{lema2} asserts $(3) \Leftrightarrow (4),$ we need just to prove $(5) \Rightarrow (4).$

Let $\phi: \mathbb{T}^2 \rightarrow \mathbb{R}$ be as in Lemma \ref{lema2}. By definition of degree

$$d = \int_{\mathbb{T}^2}f^{\ast} w =  \int_{\mathbb{T}^2} Jf(x) dw(x),$$ where $w$ is positive normalized volume form defined on $\mathbb{T}^2.$ Take $dw(x) = e^{-\phi(x)}dm(x),$ we can choose $\phi$ such that $w$ is normalized. Thus, as in Lemma \ref{lema2}

$$d =  \int_{\mathbb{T}^2}f^{\ast} w = \int_{\mathbb{T}^2} Jf(x) e^{-\phi(f(x))}dm(x) = \int_{\mathbb{T}^2} c e^{-\phi(x)}dm(x) = c , $$
we conclude $d = c$ and the Theorem \ref{teoplus} is proved.

\section{Proof of Theorem \ref{teo3} }

\begin{lemma}\label{lema1} Let $A: \mathbb{T}^3 \rightarrow \mathbb{T}^3$ a linear Anosov endomorphism as in Theorem \ref{teo3}. Then there is a $C^1$ neighborhood $\mathcal{U}$ of $A,$ such that every $f \in \mathcal{U}$ is Anosov with partially hyperbolic decomposition and moreover, if $f$ is special, then it is dynamically coherent with quasi-isometric $wu-$foliation.
\end{lemma}

\begin{proof}
It is known that $A$ is a partially hyperbolic endomorphism then its lift to $\mathbb{R}^3$ is a partially hyperbolic diffeomorphism.There is a $C^1$ neighborhood of $A$ such that every $C^1$ map  $f \in \mathcal{U}$ is an Anosov endomorphism with partially hyperbolic structure, meaning that for every $\tilde{x} = (x_n)_{n \in \mathbb{Z}} \in (\mathbb{T}^3)^f,$ holds the decomposition $T_{x_n} \mathbb{T}^3 = E^s_f \oplus E^{wu}_f \oplus E^{su}_f.$ For details of partially hyperbolic endomorphisms we refer \cite{micsa}.

If $f$ is enough $C^1$ close to $A,$ then $E^{s}_A \oplus E^{wu}_A $ is uniformly transversal to $E^{su}_{\overline{f}},$ so by Brin \cite{Br} the foliation $\mathcal{F}^{su}_{\overline{f}}$ is quasi isometric. In the same way $\mathcal{F}^{s}_{\overline{f}}$ is quasi isometric. Using again \cite{Br}, we conclude $\overline{f}$ is dynamicaly coherent and $\mathcal{F}^{wu}_{\overline{f}}$ is quasi isometric, since $E^s_A \oplus E^{su}_A$ is uniformly transversal to $E^{wu}_{\overline{f}}.$

Since $\overline{f}$ is dynamically coherent $\mathcal{F}^{cs}_f$ tangent to $E^{wu}_f \oplus E^{s}_f$ is uniquely defined on each $x \in \mathbb{T}^3.$ Of course, by \cite{micsa} the bundle $E^{wu}_f \oplus E^{s}_f$ is uniquely defined for each point $x \in \mathbb{T}^3.$ If there was a point $p$ admitting two different local tangent leaves $\mathcal{F}_{1,f}^{cs}(p)$ and $\mathcal{F}_{2,f}^{cs}(p).$  So by invariance of $E^{wu}_f \oplus E^{s}_f$ we could lift these local leaves to the same level to local leaves $\mathcal{F}_{1,\overline{f}}^{cs}(q)$ and $\mathcal{F}_{2,{\overline{f}}}^{cs}(q)$ which contradicts the dynamically coherence of $\overline{f}.$

Finally if $f$ is special the $\mathcal{F}^{wu}_f(x) = \mathcal{F}^{cs}_x \cap W^{u}_x$ is uniquely defined for each $x,$ where $W^u(x)$ is the unstable manifold tangent to $E^{wu}_f(x) \oplus E^{su}_f(x).$
\end{proof}

\begin{lemma}  Let $A: \mathbb{T}^3 \rightarrow \mathbb{T}^3$ a linear Anosov endomorphism as in Theorem \ref{teo3}. Given $\mathcal{U} $ as in Lemma \ref{lema1} and $f \in \mathcal{U},$ a special Anosov endomorphism, then the conjugacy $h$ such that $h \circ f = A \circ h$ is such that $h (\mathcal{F}^{wu}_f(x)) = \mathcal{F}^{wu}_A(h(x)).$
\end{lemma}

\begin{proof} Since $\mathcal{F}^{wu}_f$ is a quasi-isometric foliation, the proof here follows is analogous to the proof of Proposition 1 of \cite{G}.
\end{proof}

\begin{lemma}\label{lema3} Consider $A$ and $f$ as in Theorem $\ref{teo3}.$ Then there is a $C^1$ neighborhood $\mathcal{U}$ of $A$ such that for every $f \in \mathcal{U}$ a $C^r, r\geq 2,$ a special and $m-$preserving Anosov endomorphism holds the implication: if $\mathcal{F}^{wu}_f$ is absolutely continuous then $\lambda^{\sigma}_f(m) = \lambda^{\sigma}_A, \sigma \in \{s, wu, su\}.$
\end{lemma}

\begin{proof} Evidently we take $\mathcal{U} $ as in Lemma \ref{lema1}. Consider $\overline{f}: \mathbb{R}^3  \rightarrow  \mathbb{R}^3 $ the lift of $f.$ Assume that $\mathcal{F}^{wu}_f$ is absolutely continuous. Since $\mathcal{F}^{wu}_f$ is also quasi isometric, then using the same geometric strategy to prove Theorem \ref{teo1}  we can show $\lambda^{wu}_f(m) \leq \lambda^{wu}_A.$ Since $f$ is at least $C^2,$ the foliation $\mathcal{F}^{su}_{\overline{f}}$ is also absolutely continuous and quasi isometric, and thus  $\lambda^{su}_f(m) \leq \lambda^{su}_A,$ following the techniques in the proof Theorem \ref{teo1}.

Since $f$ is $m-$presering, by the Pesin formula, Theorem \ref{shi} and Lemma \ref{lema2}

$$h_m(f) = \lambda^{wu}_f(m) + \lambda^{su}_f(m) = \log(k) - \lambda^s_f(m) = \log(k) - \lambda^s_A = \lambda^{wu}_A + \lambda^{su}_A ,$$ where $k > 1$ is the degree of $f$ and $A.$

Since $\lambda^{wu}_f(m) + \lambda^{su}_f(m) =  \lambda^{wu}_A + \lambda^{su}_A,$ by the above inequalities we conclude $\lambda^{\sigma}_f(m) = \lambda^{\sigma}_A, \sigma \in \{wu, su\}.$ The identity of stable Lyapunov exponents between $f$ and $A$ comes from Theorem \ref{shi}.
\end{proof}

\begin{lemma}\label{lema4} Let $\mathcal{U} $ be as above and $f \in \mathcal{U}$ a $C^r, r \geq 2, $  special Anosov endomorphism such that $\mathcal{F}^{wu}_f$ is absolutely continuous. Then the conjugacy $h$ is at least $C^1$ on $wu-$manifolds, particularly $\lambda^{wu}_f(p) = \lambda^{wu}_A, p \in Per(f).$
\end{lemma}

\begin{proof} In this case it know that $h(\mathcal{F}^{wu}_f(x)) = \mathcal{F}^{wu}_A(h(x)).$ Since $m$ is the uniquely SRB for $A$ and $f$ respectively, then $h_{\ast}(m) = h^{-1}_{\ast}(m) = m.$

Let $\eta $ an increasing subordinate to unstable manifold $W^u_f.$ We can consider $\eta^{wu}_f$ such that $\eta^{wu}_f(\tilde{x}) = \eta(\tilde{x})\cap p^{-1}(\mathcal{F}^{wu}_f(p(\tilde{x}))).$ Since the $wu-$leaves are expanding for $f$ and submanifolds of $W^u,$ then the partition $\eta^{wu}_f$ is an increasing subordinate to $\mathcal{F}^{wu}_f$ partition, the scenery is done. Now the theory of SRB of \cite{QZ} can be applied to $\mathcal{F}^{wu}_f.$ Denote by $\eta^{wu}_A = \tilde{h}(\eta^{wu}_f),$ where $\tilde{h}$ is the induced by $h$ in the level of limit inverse spaces. By the increasing property of subordinate partitions, the values of conditional entropies
$H_{\hat{m}}(\eta^{wu}_A | \tilde{A}(\eta^{wu}_A) ), H_{\hat{m}}( \eta^{wu}_f |\tilde{f}(\eta^{wu}_f)) $ are independent of the chosen increasing subordinated to $wu-$foliation partition, see for instance Lemma 5.3, Chapter VI of \cite{LQ}. These numbers we call respectively by  $ h_{m}(A,\mathcal{ F}^{wu}_A)$ and $h_{m}(f,\mathcal{F}^{wu}_f).$

Again, since $h_{\ast}(m) = m , h(\mathcal{F}^{wu}_f(x)) = \mathcal{F}^{wu}_A(h(x)),$ by absolute continuity and the Pesin's entropy formula applied to $wu-$foliation, we obtain
$$ \lambda^{wu}_A = h_{m}(A,\mathcal{ F}^{wu}_A) = h_{m}(f,\mathcal{ F}^{wu}_f) = \lambda^{wu}_f(m). $$

As in Lemma \ref{lemasmooth} and using absolute continuity of $\mathcal{F}^{wu}_f,$ we can find $h$ on $wu-$leaves by solving analogous ordinary differential equations

$$x' = \frac{\Delta^{wu}_f( y_0 , t)}{\Delta^{wu}_A( h(y_0), x)}, x(y_0) =  h(y_0), $$
where $\Delta_f^{wu}$ has an  analogous formulation to $\Delta_f^{u},$ considering the jacobian $Jf^{wu}.$

We conclude that $h$ is $C^{1+\alpha},$ for some $\alpha > 0$ restricted to $wu-$leaves, particularly $\lambda^{wu}_f(p) = \lambda^{wu}_A, p \in Per(f).$

\end{proof}

\begin{lemma}\label{lema5} Let $\mathcal{U} $ be as above and $f \in \mathcal{U}$ a $C^r, r \geq 2, $ a $m-$preserving special Anosov endomorphism such that $\lambda^{wu}_f(p) = \lambda^{wu}_A, p \in Per(f),$ then $\mathcal{F}^{wu}_f$ is absolutely continuous.
\end{lemma}

\begin{proof}

Since $h(\mathcal{\mathcal{F}}^{wu}_f(x)) = \mathcal{F}^{wu}_A(h(x))$ and $\lambda^{wu}_f(p) = \lambda^{wu}_A, p \in Per(f),$ we conclude that $h$ is $C^1$ restricted to $wu-$leaves, see the proof of Theorem B of \cite{micendo}, by using conformal metrics. Particularly, by Theorem \ref{shi} and Lemma \ref{lema2}, $\lambda^{su}_f(p) = \lambda^{su}_A, p \in Per(f),$ particularly by using Livsic's Theorem $\lambda^{\ast}_f(x) = \lambda^{\ast}_A, \ast \in \{wu, su\}$ for any $x \in \mathbb{T}^3.$ We can conclude that $h_m(f) = \lambda^{wu}_A + \lambda^{su}_A.$ We note that if $\nu = h_{\ast}(m),$ then $h_{\nu}(A) = h_m(f) = \lambda^{wu}_A + \lambda^{su}_A,$ and then $\nu$ is a SRB measure of $A.$ By uniqueness of SRB, $\nu = m,$ and $m = h_{\ast}(m) = h^{-1}_{\ast}(m).$

Now suppose by contradiction that locally there is a set $Z,$ such that $m(Z) = 0,$ meeting the local component $\mathcal{F}^{wu}_f(x)$ of $x$ in a set of positive Lebesgue measure of this component, for every $x \in P,$ for a Borel set $P$ such that  $m(P) > 0.$ Now taking $Z' = h(Z), P' = h(P),$ since $h_{\ast}(m) = m,$ it leads us to $m(Z') = 0, m(P') > 0,$ moreover given $y \in P', y = h(x), x \in P.$ So each local component $\mathcal{F}^{wu}_A(y)$ containing $y \in P'$ meets $Z'$ in a positive Lebesgue measure set, because $h $ is $C^1$ on $wu-$leaves. But it contradicts the absolute continuity of the foliation $\mathcal{F}^{wu}_A.$ Analogously, if we suppose inside an open set $U,$ the existence of a set $P \subset U$ such that $m(U\setminus P) = 0,$ and for any $p \in P$ the component $\mathcal{F}^{wu}_f(p)$  meets $Z$ in a zero Lebesgue measure of the leaf, then we can conclude $m(Z) = 0.$

The proof of Theorem \ref{teo3} is now completed.

\end{proof}

Now considering Lemma \ref{lema2}, Theorem \ref{teo3} and taking account Theorem \ref{shi}, we get straightforwardly.

\begin{cor}\label{PD} Let $A:\mathbb{T}^3 \rightarrow \mathbb{T}^3 $ be a linear Anosov endomorphism which is irreducible over $\mathbb{Q}$ and it has three Lyapunov exponents $\lambda^s_A < 0  < \lambda^{wu}_A < \lambda^{su}_A.$ Then there is a $C^1$ neighborhood $\mathcal{U}$ of $A$ such that for a given $f \in \mathcal{U}$ being a $C^r, r\geq 2,$  special and $m-$preserving Anosov endomorphism the following are equivalent.
\begin{enumerate}
\item $\mathcal{F}^{wu}_f$ is absolutely continuous.
\item $\lambda^{wu}_f(p) = \lambda^{wu}_A,$ for any $p \in Per(f).$
\item $\lambda^{su}_f(p) = \lambda^{su}_A,$ for any $p \in Per(f).$
\item $\lambda^{\sigma}_f(p) = \lambda^{\sigma}_A, \sigma \in \{s,wu,su\},$ for any $p \in Per(f).$
\end{enumerate}

\end{cor}



\end{document}